\documentclass[a4paper,11pt]{article}%
\usepackage{amssymb}
\usepackage{amsmath}
\usepackage{amsfonts}
\usepackage{graphicx}%
\providecommand{\U}[1]{\protect\rule{.1in}{.1in}}
\providecommand{\U}[1]{\protect\rule{.1in}{.1in}}
\providecommand{\U}[1]{\protect\rule{.1in}{.1in}}
\newtheorem{theorem}{Theorem}

\newtheorem{corollary}[theorem]{Corollary}

\newenvironment{proof}[1][Proof]{\noindent\textbf{#1.} }{\ \rule{0.5em}{0.5em}}
\begin{document}

\title{The double of the doubles of Klein surfaces}
\author{Antonio F. Costa, Paola Cristofori and Ana M. Porto}
\maketitle



\textbf{Abstract.} A Klein surface is a surface with a dianalytic structure. A
double of a Klein surface $X$ is a Klein surface $Y$ such that there is a
degree two morphism (of Klein surfaces) $Y\rightarrow X$. There are many
doubles of a given Klein surface and among them the so-called natural doubles
which are: the complex double, the Schottky double and the orienting double
(see \cite{AG}, \cite{CHS}). We prove that if $X$ is a non-orientable Klein
surface with non-empty boundary, the three natural doubles, although distinct
Klein surfaces, share a common double: \textquotedblleft the double of
doubles\textquotedblright\ denoted by $DX$. We describe how to use the double
of doubles in the study of both moduli spaces and automorphisms of Klein
surfaces. Furthermore, we show that the morphism from $DX$ to $X$ is not given
by the action of an isometry group on classical surfaces.

\section{Introduction}

A (compact) Klein surface is a surface with a dianalytic structure, i. e. a
surface where the charts are defined on open sets of the upper-half complex
plane ${\mathcal{U}}$ and the transition functions are analytic or
anti-analytic (see \cite{AG}, \cite{BEGG} or \cite{N}). Topologically compact
Klein surfaces may be non-orientable and with boundary. The folding map
$\phi:{\mathbb{C}}\longrightarrow\mathcal{U}$, is defined by $\phi
(x+iy)=x+i|y|$, and a smooth morphism of Klein surfaces is a map which is
either locally complex smooth or locally the folding map, the latter occurring
over the boundary of the image (for a more precise definition see \cite{AG}).

A double of a Klein surface $X$ is a Klein surface $Y$ such that there is a
degree two morphism $Y\rightarrow X$. Three types of doubles, the so-called
natural doubles, turn out to be historically interesting: the complex double,
the Schottky double and the orienting double; they are defined in \cite{AG} in
terms of equivalence classes of dianalytic atlases. In \cite{CHS} the doubles
of Klein surfaces are studied by using subgroups of uniformizing Euclidean and
non-Euclidean crystallographic groups.

If $X$ is a non-orientable Klein surface with non-empty boundary, we prove
that the three natural doubles, although distinct Klein surfaces, share a
common double: \textquotedblleft the double of doubles\textquotedblright. The
main purpose of this paper is the study of this Riemann surface. We first
establish the relations between each natural double with the double of doubles
(section 5) and we apply this concept to the study of automorphisms of Klein
surfaces (section 7) and to the theory of real algebraic curves (section 8).
In section 6 it is shown that the morphism from the double of doubles to the
given Klein surface cannot be visualized as the natural projection on the
space of orbits produced by action of an isometry group on classical surfaces.

This article has been motivated by a question of Gareth Jones after a talk by
D. Singerman on \cite{CHS} in the Conference in honour of E. Bujalance in
Link\"{o}ping (2013).

\textbf{Aknowledgments}. This work was partially supported by MTM2014-55812
(Spain)\ and GNSAGA-INdAM (Italy).

\bigskip

\section{Klein surfaces and NEC groups}

The algebraic genus of a Klein surface $X$ of genus $g$ with $k$ boundary
components is, by definition, $2g+k-1$ if $X$ is orientable and $g+k-1$ if $X$
is non-orientable. The algebraic genus is the topological genus of the complex
double of $X$ (see section 4).

Every Klein surface has uniformization $\mathcal{S}/\Gamma$ where
$\mathcal{S}$ is a simply-connected Riemann surface and $\Gamma$ is a
crystallographic group without elliptic elements (it might have reflections
though). If the algebraic genus of the surface is greater than $1$, then
$\mathcal{S}=\mathcal{U}$, the upper complex half-plane, and $\Gamma$ is a
(planar) non-Euclidean crystallographic (NEC) group. If the algebraic genus is
equal to $1$ (for example the M\"{o}bius band) then $\mathcal{S}=\mathbb{C}$
and $\Gamma$ is a (planar) Euclidean crystallographic group. These groups are
called surface Euclidean or non-Euclidean crystallographic groups and have
assigned a signature of the form (see \cite{BEGG} and \cite{S})%
\begin{equation}
(g;\pm;[-];\{(-)^{k}\}). \label{signature}%
\end{equation}
Here, $(-)^{k}$ means $k$ empty period cycles. If this occurs, $\mathcal{S}%
/\Gamma$ is a compact surface of genus $g$ with $k$ boundary components; it is
orientable when the $+$ sign occurs and non-orientable otherwise. The group
$\Gamma$ has a fundamental region that is a Euclidean or hyperbolic polygon
$\mathcal{P}$. If the $+$ sign occurs then the fundamental region for the
group is a hyperbolic polygon with surface symbol%

\begin{equation}
\alpha_{1}\beta_{1}\alpha_{1}^{\prime}\beta_{1}^{\prime}\ldots,\,\alpha
_{g}\beta_{g}\alpha_{g}^{\prime}\beta_{g}^{\prime}\epsilon_{1}\gamma
_{1}\epsilon_{1}^{\prime}\ldots,\,\epsilon_{k}\gamma_{k}\epsilon_{k}^{\prime}
\label{surfacesymbol+}%
\end{equation}
If the $-$ sign occurs then the fundamental polygon has surface symbol
\begin{equation}
\alpha_{1}\alpha_{1}^{\ast}\ldots\alpha_{g}\alpha_{g}^{\ast}\epsilon_{1}%
\gamma_{1}\epsilon_{1}^{\prime}\ldots\epsilon_{k}\gamma_{k}\epsilon
_{k}^{\prime} \label{surfacesymbol-}%
\end{equation}

The group has two possible presentations; if the $+$ sign occurs the
presentation is%
\begin{align*}
\langle a_{1},b_{1},\ldots,a_{g},b_{g},e_{1},\ldots,e_{k},c_{1},\ldots,c_{k}
&  \mid\\
\Pi_{i=1}^{g}[a_{1},b_{i}]e_{1}\cdots e_{k}  &  =1,c_{i}^{2}=1,e_{i}c_{i}%
e_{i}^{-1}=c_{i}\text{ \ \ }(i=1,\ldots,k)\rangle
\end{align*}

Here $a_{i},b_{i}$ are translations or hyperbolic, $c_{i}$ are reflections and
$e_{i}$ are orientation-preserving though usually hyperbolic. Moreover
$a_{i}(\alpha_{i}^{\prime})=\alpha_{i},b_{i}(\beta_{i}^{^{\prime}})=\beta
_{i},e_{i}(\epsilon_{i}^{\prime})=\epsilon_{i}$ and $c_{i}$ fixes the edge
$\gamma_{i}$.

If the $-$ sign occurs the presentation is%
\begin{align*}
\langle d_{1}\ldots,d_{g},e_{1},\ldots,e_{k},c_{1},\ldots,c_{k}  &  \mid\\
d_{1}^{2}\cdots d_{g}^{2}e_{1}\cdots e_{k}  &  =1,c_{i}^{2}=1,e_{i}c_{i}%
e_{i}^{-1}=c_{i}\quad(i=1,\ldots,k)\rangle
\end{align*}

Here $d_{i}$ are glide-reflections and $d_{i}(\alpha_{i}^{\ast})=\alpha_{i}$

This type of presentations of Euclidean or NEC groups will be called
\textit{canonical presentation} and a generator will be a \textit{canonical
generator} and in both presentations the first relation is called the
\textit{long} relation.

\bigskip

\section{Standard epimorphisms of NEC\ groups and doubles of Klein surfaces}

A double of a Klein surface $X=\mathcal{S}/\Gamma$ has the form $\mathcal{S}%
/\Lambda=Y$ where $\Lambda$ is a surface subgroup of index 2 in $\Gamma$, then
there is an epimorphism $\theta:\Gamma\longrightarrow C_{2}=\langle t\mid
t^{2}=1\rangle$, with $\ker\theta=\Lambda$, called the monodromy epimorphism.

A Klein surface may have a large number of doubles (see Theorem 1 of
\cite{CHS} and \cite{H}). For this reason we focus our study on the most
important ones mentioned in \cite{AG} and \cite{CHS}.

Lets us gather the canonical generators of $\Gamma$ in sets and define:
$E=\{e_{1},\ldots,e_{k}\}$, $C=\{c_{1},\ldots,c_{k}\}$, $A=\{a_{1}%
,b_{1},\ldots,a_{g},b_{g}\}$ or $A=\{d_{1},\ldots,d_{g}\}$. We will consider
only the doubles whose monodromies $\theta:\Gamma\longrightarrow C_{2}$ are
constant on each set of generators. An epimorphism with this property is
called a standard epimorphism.

\begin{theorem}
If $k$ is even then there are 7 standard epimorphisms $\theta:\Gamma
\longrightarrow C_{2}$, while if $k$ is odd there are only 3 standard epimorphisms.
\end{theorem}

Proof in \cite{CHS}.

\bigskip

In \ref{tab1} we have listed the standard epimorphisms and the corresponding
topological type of the double $\mathcal{U}/\ker\theta$, see \cite{CHS}. We
distinguish between the cases where $\Gamma$ has orientable or non-orientable
quotient space. Here, $k$ is the number of boundary components of
$\mathcal{U}/\Gamma$, $B$ is the number of boundary components of the double
$\mathcal{U}/\ker\theta$ and the orientability of $\mathcal{U}/\ker\theta$ is
denoted by $+$ or $-$.%

\begin{equation}%
\begin{tabular}
[c]{lllll}\hline
& Standard epimorphism & Boundary & \multicolumn{2}{c}{Orientability of the
double $\mathcal{U}/\ker\theta$}\\
& $\theta$ & $B$ & $\mathcal{U}/\Gamma$ non-orientable & $\mathcal{U}/\Gamma$
orientable\\\hline\hline
1. & \multicolumn{1}{c}{$E\rightarrow\{1\}$ $C\rightarrow\{t\}$ $A\rightarrow
\{t\}$} & $0$ & \multicolumn{1}{c}{$+$} & \multicolumn{1}{c}{$-$}\\
2. & \multicolumn{1}{c}{$E\rightarrow\{1\}$ $C\rightarrow\{1\}$ $A\rightarrow
\{t\}$} & $2k$ & \multicolumn{1}{c}{$+$} & \multicolumn{1}{c}{$+$}\\
3. & \multicolumn{1}{c}{$E\rightarrow\{1\}$ $C\rightarrow\{t\}$ $A\rightarrow
\{1\}$} & $0$ & \multicolumn{1}{c}{$-$} & \multicolumn{1}{c}{$+$}\\
4. & \multicolumn{1}{c}{$E\rightarrow\{t\}$ $C\rightarrow\{1\}$ $A\rightarrow
\{1\}$} & $k$ & \multicolumn{1}{c}{$-$} & \multicolumn{1}{c}{$+$}\\
5. & \multicolumn{1}{c}{$E\rightarrow\{t\}$ $C\rightarrow\{1\}$ $A\rightarrow
\{t\}$} & $k$ & \multicolumn{1}{c}{$-$} & \multicolumn{1}{c}{$+$}\\
6. & \multicolumn{1}{c}{$E\rightarrow\{t\}$ $C\rightarrow\{t\}$ $A\rightarrow
\{1\}$} & $0$ & \multicolumn{1}{c}{$-$} & \multicolumn{1}{c}{$-$}\\
7. & \multicolumn{1}{c}{$E\rightarrow\{t\}$ $C\rightarrow\{t\}$ $A\rightarrow
\{t\}$} & $0$ & \multicolumn{1}{c}{$-$} & \multicolumn{1}{c}{$-$}\\\hline
\end{tabular}
\ \ \ \ \ \tag{Table 1}\label{tab1}%
\end{equation}

The three first rows describe the monodromies of the so-called \textit{natural
doubles}, which are the most important from the historical point of view (see
section 4 and \cite{CHS}).

\bigskip

\section{The natural doubles}

Let $X=\mathcal{U}/\Gamma$, where $\Gamma$ is a crystallographic surface group.

\begin{enumerate}
\item The complex double:
\end{enumerate}

If $X$ is a Klein surface then its complex double $X^{+}$ is the unique double
which is a Riemann surface without boundary. The complex double of
$X=\mathcal{U}/\Gamma$ is $\mathcal{U}/\Gamma^{+}$ where $\Gamma^{+}$ is the
subgroup of $\Gamma$ consisting of those transformations preserving
orientation. If $X$ is non-orientable then the generators of $A$ are glide
reflections and so the complex double is given by epimorphism 1; it
corresponds to epimorphism 3 when $X$ is orientable. The genus of the complex
double $X^{+\text{ }}$is the algebraic genus of the Klein surface $X$.

\begin{enumerate}
\item[2.] The orienting double.
\end{enumerate}

Let $X$ be a Klein surface and suppose that $\partial X$ has $k$ components.
For each $i=1,...,k$ fill in each boundary component with a disc $D_{i}$. We
get a surface $\hat{X}$ without boundary with the same orientability as $X$.
Now consider the complex double of $\hat{X}$. Let $D_{i}^{1}$ and $D_{i}^{2}$
be the lifts of $D_{i}$ to $\hat{X}$. If we remove these discs from $\hat{X}$
we end up with an orientable surface $OX$ which has $2k$ boundary components
and clearly $OX$ is an unbranched two-sheeted covering of $X$. We call $OX$
the orienting double of $X$. Note that if $X$ is orientable then $OX$ has two
connected components.

If we consider the epimorphisms of \ref{tab1} we see that we only have a
covering with twice as many boundary components as the original surface for
epimorphism 2; so this epimorphism corresponds to the orienting double of a
non-orientable Klein surface. In the case of orientable Klein surfaces the
orienting double consists of two copies of the original surface. If the
surface $X$ is non-orientable with empty boundary, the orienting double
coincides with the complex double.

\begin{enumerate}
\item[3.] The Schottky double
\end{enumerate}

Let $Y$ be a double of the Klein surface $X$. Then $Y$ admits an involution
$h\in\Gamma$ such that ${Y}/\langle h\rangle=X.$ As we are considering
unbranched but possibly folded coverings, the fixed-point set of $h$ will
include a collection of simple closed curves (see for instance \cite{BCNS}).
We define the Schottky double of $X$ to be a Klein surface $SX$ without
boundary with the same orientability as $X$ admitting a dianalytic involution
$h$ whose fixed curves separate $SX$ and such that $SX/\langle h\rangle=X$.

\begin{theorem}
[\cite{CHS}]Let $X=\mathcal{U}/\Gamma$ be a Klein surface with boundary and
$SX=\mathcal{U}/\Lambda$ its Schottky double, where $\Gamma$ and $\Lambda$ are
crystallographic surface groups. Let $\theta:\Gamma\longrightarrow
\Gamma/\Lambda\cong C_{2}$ be the natural epimorphism. Then $\theta$ is the
epimorphism 3 of \ref{tab1}.
\end{theorem}

Note that if $X=\mathcal{U}/\Gamma$ is orientable, the Schottky double
coincides with the complex double. If $X$ is non-orientable without boundary,
the Schottky double has two connected components both isomorphic to $X$.

\bigskip

\section{The double of the natural doubles}

Let $X$ be a non-orientable Klein surface with non-empty boundary. As shown in
the above section, $X$ has three different natural doubles: $X^{+},OX$ and
$SX$. The surfaces $OX$ and $SX$ are, in general, proper Klein surfaces (i.e.
non-orientable or bordered Klein surfaces). Note that $(OX)^{+}=S(OX)$ since
$OX$ is orientable and $(SX)^{+}=O(SX)$ because $SX$ has no boundary. The
following results establish that $(OX)^{+}=S(OX)=(SX)^{+}=O(SX)$ as well; this
is the Riemann surface that we shall call the double of (the natural) doubles,
and denote by $DX$.

\begin{theorem}
\label{Main}Let $X=\mathcal{U}/\Gamma$ be a non-orientable Klein surface with
non-empty boundary. Let $SX$ be the Schottky double, $OX$ be the orienting
double and $X^{+}$ be the complex double of $X$. There exists a Riemann
surface $DX$ such that $\mathrm{Aut}^{\pm}(DX)$ contains a group $\left\langle
s,t\right\rangle $ isomorphic to $C_{2}\times C_{2}$ and such that:
$DX/\left\langle s\right\rangle =OX$ is the orienting double, $DX/\left\langle
t\right\rangle =SX$ is the Schottky double of $X$, and $DX/\left\langle
st\right\rangle =X^{+}$ is the complex double.
\end{theorem}

\begin{proof}
We define $\omega:\Gamma\longrightarrow C_{2}\times C_{2}=\left\langle
s,t\right\rangle $ by:%
\[
A\rightarrow\{t\};E\rightarrow\{1\};C\rightarrow\{s\}
\]

Let $DX$ be ${\mathcal{U}}/\ker\omega$, then we have the following diagram
that proves the theorem:%
\[%
\begin{array}
[c]{ccccc}
&  & DX={\mathcal{U}}/\ker\omega &  & \\
& \swarrow & \downarrow & \searrow & \\
X^{+}={\mathcal{U}}/\omega^{-1}(\left\langle st\right\rangle ) &  &
OX={\mathcal{U}}/\omega^{-1}(\left\langle s\right\rangle ) &  &
SX={\mathcal{U}}/\omega^{-1}(\left\langle t\right\rangle )\\
& \searrow & \downarrow & \swarrow & \\
&  & X &  &
\end{array}
\]
\bigskip
\end{proof}

If $X$ is non-orientable, has genus $g$ and $k$ boundary components then (see
Table 1 and \cite{CHS}) the complex double is an (orientable) Riemann surface
(without boundary) of genus $g+k-1$, the orienting double is an orientable
Klein surface of genus $g-1$ with $2k$ boundary components, the Schottky
double is a non-orientable Klein surface without boundary of genus $2g+2k-2$
and, finally, the double of doubles of $X$ is an (orientable) Riemann surface
(without boundary) of genus $2g+2k-3$.

Note that $st$ is an orientation preserving element while $s$ and $t$ are
orientation reversing.

\begin{corollary}
Given a non-orientable Klein surface with non-empty boundary, the complex
double $(SX)^{+}=DX$ of the Schottky double $SX=DX/\left\langle t\right\rangle
$ of $X$ coincides with the complex double $(OX)^{+}=DX$ of the orienting
double $OX=DX/\left\langle s\right\rangle $ of $X$. The anticonformal
involution $s$ is fixed point free and the fixed point set of $t$ is
separating. The conformal involution $st$ is fixed point free.
\end{corollary}

\begin{proof}
Remember that if $X$ is a Klein surface then its complex double $X^{+}$ is the
unique double which is a Riemann surface without boundary. The double of
doubles $DX$ is a Riemann surface and there are degree two morphisms from $DX$
to $OX$ or $SX$ thus $(OX)^{+}=(SX)^{+}=DX$. Since $SX$ has no boundary then
$s$ is fixed point free and since $OX$ is orientable $DX-Fix(t)$ has two
connected components. Finally, as $DX\rightarrow X^{+}$ is an order two
morphism between Klein surfaces which are in fact Riemann surfaces, it is an
unbranched two fold covering as well and $st$ is fixed point free.
\end{proof}

The description of the unbranched covering $DX\rightarrow X^{+}$ is the following:

\begin{theorem}
Let $\sigma$ be the anticonformal involution given by $X^{+}\rightarrow X$ and
$Fix(\sigma)$ be the fixed point set of $\sigma$. Let $\left\langle
.,.\right\rangle $ be the intersection form in $H_{1}(X^{+},\mathbb{Z}_{2})$
and $[Fix(\sigma)]$ be the cycle in $H_{1}(X^{+},\mathbb{Z}_{2})$ represented
by the union of the curves in $Fix(\sigma)$. The covering $DX\rightarrow
DX/\left\langle st\right\rangle =X^{+}$ is an unbranched covering with
monodromy
\begin{align*}
\left\langle \lbrack Fix(\sigma)],.\right\rangle  &  :\pi_{1}(X^{+}%
)\rightarrow H_{1}(X^{+},\mathbb{Z}_{2})\rightarrow C_{2}\\
\gamma &  \mapsto\lbrack\gamma]\mapsto\left\langle \lbrack Fix(\sigma
)],[\gamma]\right\rangle
\end{align*}

\end{theorem}

\begin{proof}
Let us restrict our proof to the case where $X$ is a Klein surface with
algebraic genus $>1$, then $\Gamma$ is NEC\ group. The group uniformizing
$X^{+}$ is $\Gamma^{+}$ and the monodromy of the covering $DX\rightarrow
X^{+}$ is just the restriction of $\omega$ to $\Gamma^{+}=\omega
^{-1}(\left\langle st\right\rangle )$ which is an epimorphism on $\left\langle
st\right\rangle =C_{2}$. Let $g\in\Gamma^{+}$ such that $\omega(g)=$ $st$ and
let $w$ be an expression of $g$ as an irreducible word in some canonical set
of generators of $\Gamma$. Then an odd number of reflections appears in the
word $w$. If $\gamma$ is the curve that is the projection on $X^{+}$ of the
axis of the hyperbolic element $g$, then $\gamma$ cuts $Fix(\sigma)$ in an odd
number of points. Hence $\left\langle [Fix(\sigma)],[\gamma]\right\rangle
\neq0$. In similar way if $\omega(g)=1$ then $\left\langle [Fix(\sigma
)],[\gamma]\right\rangle =0$. Thus $\left\langle [Fix(\sigma)],.\right\rangle
$ is given by the restriction of $\omega$ to $\Gamma^{+}$, so it is the
monodromy of $DX\rightarrow X^{+}$.
\end{proof}

Note that $DX\rightarrow X^{+}$ is something living completely in the theory
of Riemann surfaces and that is naturally given by the non-orientable bordered
Klein surface $X$.

\bigskip

\section{The automorphism group of $DX\rightarrow X$ cannot be visualized in
$\mathbb{R}^{3}$.}

Every smooth surface in the Euclidean space can be made into a Riemann surface
in a natural way by restriction of the Euclidean metric to it. These surfaces
are called classical Riemann surfaces and they are considered by Beltrami and
Klein (see the introduction of \cite{G} and chapter II, section 5 of
\cite{AS}). There are some automorphisms of Riemann surfaces that can be
represented by the restriction to classical Riemann surfaces of isometries of
the Euclidean space. This is a natural way of visualizing automorphisms of
Riemann surfaces.

Note that each one of the three automorphisms $s,t,st\,$\ of the preceding
section are representable as restriction of isometries to classical Riemann
surfaces (see \cite{C}), but the complete group action of $C_{2}\times C_{2}$
is not the restriction of a finite group of isometries, so it
\textquotedblleft cannot be visualized\textquotedblright.

\textit{Example 6.1}. If $M$ is a M\"{o}bius band we know that $DM$ is an
analytical torus conformally equivalent to a classical torus $T_{1}$ embedded
in $\mathbb{R}^{3}$, such that $T_{1}$ is invariant by an order two rotation
$r$ with axis non-cutting $T_{1}$ and that the unbranched covering
$DM\rightarrow M^{+}$ is analytically equivalent to $T_{1}\rightarrow
T_{1}/\left\langle r\right\rangle $. Analogously there are embedded tori
$T_{2}$ and $T_{3}$ such that $T_{2}$ is invariant by a plane reflection $p$
with $T_{2}\rightarrow T_{2}/\left\langle p\right\rangle $ equivalent to
$DM\rightarrow OM$ and $T_{3}$ is invariant by a central symmetry $c$ such
that $T_{3}\rightarrow T_{3}/\left\langle c\right\rangle $ is equivalent to
$DM\rightarrow SM$. But there is no embedded torus $T$ and no group of
isometries $G$, isomorphic to $C_{2}\times C_{2}$, such that $T\rightarrow
T/G$ is equivalent to $DM\rightarrow M$. The obstruction is of topological
nature. Assume that we have such classical torus $T$ and a group of isometries
$G$ such that $T\rightarrow T/G$ is equivalent to $DM\rightarrow M $:\ the
group $G$ must be generated by a plane reflection and a central symmetry.
Furthermore the order two rotation $r$ of $G$ must not cut the torus $T$
because $T\rightarrow T/\left\langle r\right\rangle $ is equivalent to
$DM\rightarrow M^{+}$. The plane of symmetry must be orthogonal to the axis of
$r$, then $T/G$ is homeomorphic to a cylinder and not a M\"{o}bius band. Thus
$T\rightarrow T/G$ is not equivalent to $DM\rightarrow M$.

\bigskip

\section{The double of doubles and automorphisms}

Doubles of Klein surfaces are useful for the study of the automorphism groups
of Klein surfaces. Every automorphism of a given Klein surface $X$ lifts to an
automorphism of the complex double $X^{+}$, and in this way it is possible to
study the automorphisms of Klein surfaces by using automorphisms of Riemann
surfaces. The difficulty arises when some of the automorphisms in $X^{+}$ are
not liftings of automorphisms of $X$ and then $\mathrm{Aut}(X^{+})$ is not
isomorphic to $C_{2}\times\mathrm{Aut}(X)$. This difficulty remains, even in
case of maximal symmetry, when considering the double of doubles for
non-orientable Klein surfaces, as shown in the last example of this section;
nevertheless, the information the automorphisms of $DX$ may provide is better
than the one given by $\mathrm{Aut}(X^{+})$. This claim is supported by the
fact that although not every automorphism of $X^{+} $ lifts to $DX$ (see first
example of this section), this is true for the automorphisms of $X$ (next theorem).

\begin{theorem}
Let $X$ be a bordered non-orientable Klein surface $X$ and let $DX$ be the
double of the doubles of $X$. Then every automorphism of $X$ lifts to an
automorphism of $\mathrm{Aut}^{\pm}(DX)$ and $\mathrm{Aut}^{\pm}(DX)$ contains
a group isomorphic to $\mathrm{Aut}X\times C_{2}\times C_{2}$.
\end{theorem}

\begin{proof}
We shall prove the result for the case of surfaces $X$ of genus $>1$. Let
$\Gamma$ be a surface NEC\ group such that $X=\mathcal{U}/\Gamma$ and $\Delta$
be such that $\Gamma\vartriangleleft\Delta$ and $\Delta/\Gamma$ is isomorphic
to $\mathrm{Aut}X$. Let $\theta:\Delta\rightarrow\Delta/\Gamma\simeq
\mathrm{Aut}X$ be the natural map and $\left\langle S:R\right\rangle $ be a
canonical presentation of $\Delta$ (see, for instance, \cite{BEGG} page 14).

Let us define $\theta^{\prime}:\Delta\rightarrow\mathrm{Aut}X\times
C_{2}\times C_{2}$, by:%
\[
\theta^{\prime}(s)=(\theta(s),\theta_{2}^{\prime}(s),\theta_{3}^{\prime}(s))
\]
where $s\in S$ is a generator of the canonical presentation of $\Delta$,
$\theta_{2}^{\prime}(s)\neq1$ if and only if $s$ is orientation reversing and
$\theta_{3}^{\prime}(s)\neq1$ if and only if $s$ is a reflection in $\Gamma$.

Let us see that $\theta_{3}^{\prime}:\Delta\rightarrow C_{2}$ is a
homomorphism: the relations in $R$ which contain reflections have either the
form $e_{i}^{-1}c_{i0}e_{i}=c_{is_{i}}$, $c_{ij}^{2}=1$ or $(c_{i,j-1}%
c_{i,j})^{n_{ij}}=1$. Since $C_{2}$ is abelian and $\Gamma\vartriangleleft
\Delta$, the relations of the two first types are automatically respected by
$\theta_{3}^{\prime}$ . In the third type, if $n_{ij}$ is even, the relations
are respected by $\theta_{3}^{\prime}$ because $\theta_{3}^{\prime}%
(\Delta)=C_{2}$; if $n_{ij}$ is odd the relation $(c_{i,j-1}c_{i,j})^{n_{ij}%
}=1$ tells us that $c_{i,j-1}$ and $c_{i,j}$ are conjugate and thus either
both $c_{i,j-1}$ and $c_{i,j}$ belong to $\Gamma$ or none of them is in
$\Gamma$; in any case, $\theta_{3}^{\prime}(c_{i,j-1}c_{i,j})=1$ and the
relation is also respected.

Note that $\ker\theta=\Gamma$ uniformizes $X$, $\ker(\theta,\theta_{2}%
^{\prime})$ uniformizes $X^{+}$ and $\ker\theta^{\prime}=\ker(\theta
,\theta_{2}^{\prime},\theta_{3}^{\prime})$ uniformizes a two fold covering of
$X^{+}$. The monodromy $\omega:\ker(\theta,\theta_{2}^{\prime})=\pi_{1}%
(X^{+})\rightarrow C_{2}$ of $\mathcal{U}/\ker\theta^{\prime}\rightarrow
X^{+}$ is given by the following rule: if $\gamma\in\ker(\theta,\theta
_{2}^{\prime})$, $\omega(\gamma)\neq1$ if and only if $\gamma$ can be
expressed as a word $w_{S}$ in the system of generators $S$ of the canonical
presentation of $\Delta$, such that there is an odd number of reflections
conjugate to reflections of $\Gamma$. And this is exactely the monodromy of
$DX\rightarrow X^{+}$ by Theorem 5.

Since $\ker\theta^{\prime}$ uniformizes $DX$, every automorphism of $X$ admits
a lifting to $DX$ and $\mathrm{Aut}^{\pm}(DX)$ contains a group isomorphic to
$\mathrm{Aut}X\times C_{2}\times C_{2}$.
\end{proof}

As a consequence, an automorphism of $X^{+}$ not lifting to an automorphism of
$DX$, cannot be itself a lift of an automorphism of $X$, meaning that the
automorphisms of $DX$ provide better information on $\mathrm{Aut}(X)$ than
$\mathrm{Aut}^{\pm}(X^{+})$ do. Next example illustrates this situation:

\textit{Example 8.1}. Let $\Delta$ be a maximal NEC group with signature
\[
(1;+;[3];\{(3)\})
\]
and
\[
\left\langle a,b,x,c_{0},c_{1},e:xeaba^{-1}b^{-1}=1,x^{3}=1,c_{0}^{2}%
=c_{1}^{2}=1,(c_{0}c_{1})^{3}=1,ec_{0}e^{-1}=c_{1}\right\rangle
\]
be a canonical presentation of $\Delta$. Let us consider the epimorphism:%
\[
\theta:\Delta\rightarrow D_{3}=\left\langle s,t:s^{2}=t^{2}=(st)^{3}%
=1\right\rangle
\]
defined by:%
\begin{align*}
\theta(a)  &  =\theta(b)=1;\theta(x)=st;\theta(e)=ts\\
\theta(c_{0})  &  =s;\theta(c_{1})=tst
\end{align*}
The NEC\ group $\theta^{-1}(\left\langle s\right\rangle )$ is a non-orientable
surface crystallographic group with signature $(7;-;[-];\{(-)\})$; so
$X=\mathcal{U}/\theta^{-1}(s)$ is a Klein surface. The complex double $X^{+}$
is uniformized by $\ker\theta$ and its automorphism group is $\Delta
/\ker\theta=D_{3}$ (note that we have assumed $\Delta$ maximal).

The group $\theta^{-1}(s)$ is not normal in $\Delta,$ so the automorphism
group of $X$ is trivial, but $\mathrm{Aut}^{\pm}(X^{+})=D_{3}$, thus
$\mathrm{Aut}^{\pm}(X^{+})\gvertneqq\mathrm{Aut}(X)\times C_{2}$. The
anticonformal involution $s$ of $X^{+}$ producing $X$ as quotient has a
connected closed curve $\gamma$ as fixed point set. We will call $st$ an order
three conformal automorphism of $X^{+}$. The automorphism $st$ does not lift
to $DX$. A reason for that is as follows: the curve $\partial X$ lifts to a
closed curve with two connected components in $DX$ and to $\gamma$ in $X^{+}$,
but $st(\gamma)$ cuts $\gamma$ just in the only fixed point of $st$ which
projects on the boundary of $X$. Therefore, $st(\gamma)$ lifts to a connected
curve of $DX$ and this fact prevents the existence of a lift of $st$ to $DX$.

Finally, in the next example we show that, in some cases, the information on
$\mathrm{Aut}(X)$ provided by $\mathrm{Aut}(DX)$ is not essentially better
than the one obtained by $\mathrm{Aut}^{\pm}(X^{+})$. In fact, in \cite{BCGS}
and \cite{M}, it is established that if a bordered Klein surface $X$ has
maximal symmetry or \textquotedblleft almost maximal
symmetry\textquotedblright\ (more concretely if $\#\mathrm{Aut}(X)\geq8(p-1)
$, where $p$ is the algebraic genus of $X$) then there is a finite number of
Klein surfaces where $\mathrm{Aut}^{\pm}(X^{+})$ contains properly
$C_{2}\times\mathrm{Aut}(X)$. When $X$ is non-orientable with boundary the
first occurrence of such situation is described in the following example:

\textit{Example 8.2}. We shall describe a Klein surface $P2$ which is
topologically a projective plane with two holes and such that $\mathrm{Aut}%
^{\pm}(P2^{+})\neq C_{2}\times\mathrm{Aut}(P2)$ and $\mathrm{Aut}(DP2)\neq
C_{2}\times C_{2}\times\mathrm{Aut}(P2)$. The surface $P2$ can be uniformized
by a NEC group $\Gamma$ whose fundamental region is a regular right angled
hyperbolic octagon $O$ and the elements of $\Gamma$ produce a pairwise
identification of the sides of $O$ given by the following symbol:
\[
\alpha_{1}\gamma_{1}\alpha_{2}\gamma_{2}\alpha_{1}^{\ast}\gamma_{1}^{\prime
}\alpha_{2}^{\ast}\gamma_{2}^{^{\prime}}%
\]
where $\alpha_{i}$ is identifyied with $\alpha_{i}^{\ast}$ by a hyperbolic
glide reflection $d_{i}$, $i=1,2,\,$and $\gamma_{1}\cup\gamma_{1}^{\prime}$,
$\gamma_{2}\cup\gamma_{2}^{\prime}$ give rise to the two components of
$\partial P2$, i. e. for each $i=1,2$, $\gamma_{i}$,$\gamma_{i}^{\prime}$, are
in the fixed point set of reflections of $\Gamma$. The automorphisms group of
$P2$ is $D_{4}$ then $\mathrm{Aut}(P2)$ has of order $8$.

Now $P2^{+}$ is uniformized by the surface Fuchsian group $\Gamma^{+}$ and the
regular octagon in $P2$ lifts to a regular map $\{8,4\}$ in $P2^{+}$. Note
that there is only a regular map of type $\{8,4\}$ in surfaces of genus 2:
($R2.3^{\prime}$ following the notation in \cite{CD}). Then $P2^{+}$ is the
underlying Riemann surface in the regular map $R2.3^{\prime}$. Since such a
map can be obtained as a stellation of the regular map $R2.1$ of type
$\{3,8\}$, the Riemann surface $P2^{+}$ is also the surface underlying such
map. The group of automorphisms of $P2^{+}$ is the group of automophisms of
$R2.1$, hence the group of automorphisms of $P2^{+}$ is a $C_{2}-$extension of
$GL(2,3)$ and has $96$ elements (see Theorem B of \cite{BCGS}), so
$\mathrm{Aut}^{\pm}(P2^{+})\neq C_{2}\times\mathrm{Aut}(P2)$.

The double of doubles $DP2$ is a two fold covering of $P2^{+}$ and the map
$R2.1$ lifts to a regular map of type $\{3,8\}$. Since there is only one
regular map on genus three surfaces of type $\{3,8\}$:$\,R3.2$ (see
\cite{CD}), $DP2$ is the genus three Riemann surface underlying $R3.2$ (the
dual of the Dyck map). The group $\mathrm{Aut}(DP2)$ is the full symmetry
group of the map $R3.2$ and $\#\mathrm{Aut}(DP2)=192$ ($\#\mathrm{Aut}%
^{+}(DP2)=96,$ see for instance \cite{KK}). Hence $\mathrm{Aut}(DP2)\neq
C_{2}\times C_{2}\times\mathrm{Aut}(P2)$.

\bigskip

\section{An application to the study of the moduli space of non-separating
real algebraic curves}

The complexification $C_{\mathbb{C}}$ of a (smooth projective) \textit{real}
algebraic curve $C$ is a \textit{complex} algebraic curve, thus a compact
Riemann surface. The conjugation provides an anticonformal involution $\sigma$
on $C_{\mathbb{C}}$ and the pair $(C_{\mathbb{C}},\sigma)$ determines
completely the real curve $C$. The pair $(C_{\mathbb{C}},\sigma)$ is given by
the Klein surface $K_{C}=C_{\mathbb{C}}/\left\langle \sigma\right\rangle $, so
$K_{C}$ represents the real algebraic curve, too. The topological type $t$ of
$K_{C}$ is $(h;\pm;k)$, where $h$ is the topological genus, the sign $\pm$ is
given by the orientability and $k$ is the number of connected components of
$\partial K_{C}$. The complexification $C_{\mathbb{C}}$ is, in fact, the
complex double of $K_{C}$ and the genus of $C_{\mathbb{C}}$ is the algebraic
genus of $K_{C}$.

Assume that $K_{C}$ is non-orientable, i. e. the topological type is $(h;-;k)
$; the fixed point set $Fix(\sigma)$ of the involution $\sigma$ does not
separate the Riemann surface $C_{\mathbb{C}}$, which is why $C$ is called a
non-separating real algebraic curve.

The spaces of deformations or moduli spaces are important tools in the study
of algebraic curves. There is a different moduli space for each topological
type of Klein surfaces, i. e. once the genus of the complexification of the
real algebraic curve is fixed, the space of deformations for real
non-separating algebraic curves with algebraic genus $g$ is the disjoint
union:%
\[
\mathcal{M}_{g}^{\mathbb{R},-}\mathcal{=}%
{\displaystyle\bigcup\limits_{h+k-1=g}}
\mathcal{M}_{(h;-;k)}^{K}%
\]
where $\mathcal{M}_{(h;-;k)}^{K}$ is the moduli space of Klein surfaces with
topological type $(h;-;k)$ (see for instance \cite{Se2}, \cite{BEGG}, \cite{N}).

In some situations it is important to have a common space to relate the
different topological types of real curves with the same complexification. The
set $\mathcal{M}_{g}^{\mathbb{R},-,\mathbb{C}}$ corresponds to the set of
points in $\mathcal{M}_{g}$ that are Riemann surfaces having a non-separating
anticonformal involution. Then
\[
\mathcal{M}_{g}^{\mathbb{R},-,\mathbb{C}}=%
{\displaystyle\bigcup\limits_{0\leq k\leq g}}
\mathcal{M}_{g}^{(-,k)}%
\]
where $\mathcal{M}_{g}^{(-,k)}$ is the set of points in $\mathcal{M}_{g}$
corresponding to Riemann surfaces with an anticonformal involution of
topological type $t=(-,k)$. This space has been studied by many authors:
\cite{Se1}, \cite{N}, \cite{BCI}, \cite{CI}. Now a real non-separating
algebraic curve is a pair $(X,\sigma)$ where $X\in\mathcal{M}_{g}%
^{\mathbb{R},-}$ and $\sigma$ is an anticonformal involution of the Riemann
surface $X$, and $X/\left\langle \sigma\right\rangle $ is non-orientable. The
map $\phi:\mathcal{M}_{g}^{\mathbb{R},-}\rightarrow\mathcal{M}_{g}%
^{\mathbb{R},-,\mathbb{C}}$ given by $\phi(K)=K^{+}$ is continuous and $\phi$
restricted to $\mathcal{M}_{(h;-;k)}^{K}$ is an (orbifold) embedding, for
$h+k-1=g$ (see Corollary 8.9 of \cite{MS} and \cite{BCNS}).

Now by using the results of the above sections we obtain that real
non-separating algebraic curves of algebraic genus $g$ are in the intersection
of just two connected real analytic spaces in $\mathcal{M}_{2g-1}$.

\begin{theorem}
Let $\mathcal{M}^{(+,0)}$ (respectively $\mathcal{M}^{(-,0)}$) be the set in
$\mathcal{M}_{2g-1}$ consisting of the surfaces admitting a fixed point free
conformal (resp. anticonformal) involution. We define
\[
\mathcal{N}_{g}^{\mathbb{R},-}=\mathcal{M}^{(+,0)}\cap\mathcal{M}^{(-,0)}%
\]

Then there exists a continuous map $\psi:\mathcal{M}_{g}^{\mathbb{R}%
,-}\rightarrow\mathcal{N}_{g}^{\mathbb{R},-}$, and $\psi\circ\phi$ restricted
to $\mathcal{M}_{g}^{(-,k)}$ is an embedding, for each $k=0,...,g$.
\end{theorem}

\begin{proof}
A point in $\mathcal{M}_{g}^{\mathbb{R},-}$ may be represented by a
non-orientable Klein surface $K$. Let $\psi(K)$ be the point in $\mathcal{M}%
_{2g-1}$ given by $DK$. In the case $g>1$, $K$ is uniformized by a surface NEC
group $\Gamma$ and then the surface $\psi(K)$ is uniformized by the subgroup
$\ker\omega$, where $\omega:\Gamma\longrightarrow C_{2}\times C_{2}$ is the
epimorphism defined in the proof of Theorem \ref{Main}. If $T_{(g,k,-)}$ and
$T_{2g-1}$ are respectively the Teichm\"{u}ller spaces of Klein surfaces with
topological type $(g,k,-)$ and of Riemann surfaces of genus $2g-1$, then the
inclusion $\ker\omega\rightarrow\Gamma$ produces an isometric embedding from
$\varphi:T_{(g,k,-)}\rightarrow T_{2g-1}$ (Corollary 8.9 of \cite{MS}). The
map $\varphi$ produces the continuous map $\psi:\mathcal{M}_{g}^{\mathbb{R}%
,-}\rightarrow\mathcal{M}_{2g-1}$ sending each Klein surface to its double of
doubles $DK$. Since the double of doubles admits the action of a free
orientation preserving involution (producing as orbit space $K^{+}$) and a
free anticonformal involution (producing $SK$) we have that $\psi
(\mathcal{M}_{g}^{\mathbb{R},-})\subset\mathcal{N}_{g}^{\mathbb{R},-}$.
\end{proof}

Next result follows from the above together with Theorems 3.1 and 3.3 of
\cite{CI}

\begin{corollary}
$\psi(\mathcal{M}_{g}^{\mathbb{R},-})$ is connected.
\end{corollary}

Antonio F. Costa, Departamento de Matem\'{a}ticas Fundamentales, UNED, Madrid
28040, Spain

acosta@mat.uned.es

Paola Cristofori, Dipartimento di Scienze Fisiche, Informatiche, Matematiche,
Universit\`{a} degli studi di Modena e Reggio Emilia, Via Campi 213/A - 41125
Modena, Italy

paola.cristofori@unimore.it

Ana M. Porto, Departamento de Matem\'{a}ticas Fundamentales, UNED, Madrid
28040, Spain

asilva@mat.uned.es

\end{document}